\documentclass{amsart}

\title{Picard groups and the $K$-theory of curves with cuspidal singularities}
\author{Vigleik Angeltveit}
\address{Mathematical Sciences Institute \\
Australian National University \\
Canberra, ACT 0200 \\
Australia}

\usepackage{amsxtra}
\usepackage{amsfonts}
\usepackage[latin1]{inputenc}
\usepackage{graphicx}
\usepackage{amsmath,amssymb,latexsym,amsthm,mathrsfs}
\usepackage[all]{xy}
\usepackage{pstricks}
\usepackage{verbatim}

\newtheorem{theorem}{Theorem}[section]
\newtheorem{thm}[theorem]{Theorem}

\newtheorem{cor}[theorem]{Corollary}

\newtheorem{prop}[theorem]{Proposition}

\theoremstyle{definition}
\newtheorem{defn}[theorem]{Definition}
\newtheorem{remark}[theorem]{Remark}

\numberwithin{equation}{section}

  \newcommand{\cC}{\mathcal{C}}          \newcommand{\cM}{\mathcal{M}}

  \newcommand{\bC}{\mathbb{C}}   \newcommand{\bF}{\mathbb{F}}        \newcommand{\bN}{\mathbb{N}}  
\newcommand{\bQ}{\mathbb{Q}}      \newcommand{\bW}{\mathbb{W}}   \newcommand{\bZ}{\mathbb{Z}}

\newcommand{\sma}{\wedge} 

\newcommand{\xto}{\xrightarrow}

\newcommand{\wt}{\widetilde}
\newcommand{\wh}{\widehat}

\newcommand{\TR}{\textnormal{TR}}
\newcommand{\TC}{\textnormal{TC}}

\newcommand{\cof}{\textnormal{Cof}}
\newcommand{\pic}{\textnormal{Pic}}
\newcommand{\rk}{\textnormal{rk}}

\begin{document}

\begin{abstract}
We calculate the algebraic $K$-theory of the coordinate ring of a planar cuspidal curve over a regular $\bF_p$-algebra, thereby verifying a conjecture due to Hesselholt \cite{He14}. In the course of the proof we compute the Picard group of the homotopy category of $p$-complete genuine $C_{p^n}$-spectra.
\end{abstract}

\maketitle

\section{Introduction}
The main purpose of this paper is to provide a proof of a conjecture of Hesselholt \cite{He14} regarding the algebraic $K$-theory of planar cuspidal curves. A key ingredient in the proof is a calculation of the Picard group of the homotopy category of $p$-complete $C_{p^n}$-spectra, which lets us recognize representation spheres from homological data. The same proof technique lets us compute some other $K$-groups with less geometric input, giving new proofs of some existing results. The Picard group calculation might be of independent interest, and the reader who is mostly interested that is adviced to skip directly to Section \ref{s:picard}.

Since Hesselholt went through the $K$-theory calculation carefully, we will not repeat the details of the calculation here. Instead we will focus on proving a version of \cite[Conjecture B]{He14}, which is the key input to make Hesselholt's calculation work. Note that we do not actually prove Conjecture B but rather a stable $p$-complete $C_{p^n}$-equivariant version of it. The reader should keep a copy of Hesselholt's paper nearby. We prove the following theorem, which was stated in \cite[Theorem A]{He14} with the additional assumption that Conjecture B holds:

\begin{thm} \label{t:main}
Let $k$ be a regular $\bF_p$-algebra and let $A = k[x,y]/(x^b-y^a)$ with $a, b \geq 2$ relatively prime. Suppose $p \nmid a$. Then there is a long exact sequence
\begin{multline*}
 \ldots \to \bigoplus_r \bW_{S/b} \Omega_k^{q-2r}/V_a \bW_{S/ab} \Omega_k^{q-2r} \xto{V_b} \bigoplus_r \bW_S \Omega_k^{q-2r}/V_a \bW_{S/a} \Omega_k^{q-2r} \\ \to K_q(A, \mathfrak{a}) 
 \to \bigoplus_r \bW_{S/b} \Omega_k^{q-2r-1}/V_a \bW_{S/ab} \Omega_k^{q-2r-1} \xto{V_b} \ldots
\end{multline*}
where $S=S(a,b,r) = \{m \in \bN \quad | \quad \ell(a,b,m) \leq r\}$ and $\ell(a,b,m)$ is the number of ways to write $m = ai+bj$ with $i, j > 0$.
\end{thm}

Here $\mathfrak{a} = (x,y)$, and $\bW_S(-)$ denotes Witt vectors indexed by the truncation set $S$. The assumption that $p \nmid a$ does not cause any loss of generality, since either $p \nmid a$ or $p \nmid b$. An alternative formulation is that up to extensions the $K$-groups are given by the homology of the chain complex associated to the square
\[ \xymatrix{
 \bigoplus_r \bW_{S/ab} \Omega_k^{*-2r} \ar[r]^-{V_a} \ar[d]_{V_b} & \bigoplus_r \bW_{S/b} \Omega_k^{*-2r} \ar[d]^{V_b} \\
 \bigoplus_r \bW_{S/a} \Omega_k^{*-2r} \ar[r]^-{V_a} & \bigoplus_r \bW_S \Omega_k^{*-2r},
}\]
where we interpret $\Omega^{*-2r}_k$ as zero if $*-2r < 0$.

With additional assumptions on $k$ the answer simplifies further: If in addition $p \nmid b$ then $V_b$ is injective and the $K$-groups are given by a quotient by $V_a$ and $V_b$. Similarly, if $k$ is a perfect field of characteristic $p$ then $V_b$ is injective and $\Omega_k^r = 0$ for $r > 0$ so the result simplifies even more.

We can also say something about what happens over the integers:
\begin{thm} \label{t:mainZ}
Let $A = \bZ[x,y]/(x^b-y^a)$ with $a, b \geq 2$ relatively prime. Then $K_0(A, \mathfrak{a})$ is free of rank $(a-1)(b-1)/2$, and for $i \geq 1$ we have the following:
\begin{itemize}
\item The abelian group $K_{2i-1}(A, \mathfrak{a})$ is finite, of order
\[
 \frac{i!^{(a-1)(b-1)} i! (abi)!}{i^{(a-1)(b-1)/2} M_i (ai)! (bi)!}
\]
\item The abelian group $K_{2i}(A, \mathfrak{a})$ is free of rank $(a-1)(b-1)$.
\end{itemize}
Here $M_i$ is the product of the $(a-1)(b-1)/2$ natural numbers $k$ which satisfy $k \leq abi$ and $\ell(a,b,k) = i$.
\end{thm}
\noindent
For example, with $a=3$ and $b=5$ this says that $K_1(A, \mathfrak{a})$ has order $3^4 \cdot 5^2 \cdot 2 \cdot 4 \cdot 7$. We can interpret some of the factors in $|K_{2i-1}(A, \mathfrak{a})|$ as follows: $\frac{(abi)!}{M_i} = \prod_{\ell(a,b,k) < i} k$, $(ai)! = \prod_{\ell(a,b,bk) < i} k$, $(bi)! = \prod_{\ell(a,b,ak) < i} k$ and $i! = \prod_{\ell(a,b,abk) < i} k$.

Theorem \ref{t:main} should be compared to the calculation of the $K$-theory of $k[x]/(x^a)$ in Hesselholt and Madsen's papers \cite{HeMa97b} and \cite{HeMa99} and to the calculation of the $K$-theory of $k[x,y]/(xy)$ in Hesselholt's paper \cite{He07}. In fact, our method also gives an alternative proof of \cite[Theorem A]{HeMa99} which avoids the use of trigonometric moment curves by using Theorem \ref{t:homologyoftruncated} as input instead. Theorem \ref{t:mainZ} should be compared to the $K$-theory calculations over the integers in \cite{AnGeHe09} and \cite{AnGe11}.

The proof uses the cyclotomic trace map $K(A, \mathfrak{a}) \to \TC(A, \mathfrak{a}; p)$ to the $p$-typical version of the topological cyclic homology spectrum. We assume the reader is familiar with the definitions of topological Hochschild homology and topological cyclic homology, but we note that because $\TC(A, \mathfrak{a}; p)$ is defined as the homotopy inverse limit (over the restriction and Frobenius maps) of the $C_{p^n}$-fixed point spectrum of $THH(A, \mathfrak{a})$ as $n$ goes to infinity, it suffices to understand $THH(A, \mathfrak{a})$ as a $C_{p^n}$-spectrum.

\subsection*{Outline}
We start in Section \ref{s:THH} by recalling some facts about $THH(A)$ for $A=k[x,y]/(x^b-y^a)$ from \cite{He14}. In Section \ref{s:homology} we give a calculation of the homology groups of $\Sigma(a,b,m)$, together with their $C_m$-action. The homology of $S^1_+ \sma_{C_m} \Sigma(a,b,m)_+$ can be interpreted as part of the Hochschild homology of $A$, and we have lots of tools for studying Hochschild homology. For example, see \cite[Section 5]{He14} for a calculation of $HH_*(A)$. Understanding the homology of $\Sigma(a,b,m)$ is more difficult, especially because we also want to determine the $C_m$-action. We find a homological algebra interpretation of the homology of $\Sigma(a,b,m)$, but understanding the $C_m$-action still requires finding explicit representatives in a certain cyclic bar complex. The proof of Theorem \ref{t:homologyofsigma} is rather tedious, but we do not know how to avoid it.

In Section \ref{s:picard} we compute the Picard group of the homotopy category of $p$-complete $C_{p^n}$-spectra. The Picard group is the group of invertible objects up to isomorphism. For example, if $\alpha$ is a virtual $G$-representation then $S^\alpha$ is invertible with inverse $S^{-\alpha}$. If $\alpha = [V] - [W]$, we write $S^\alpha$ for $\Sigma^{-W} \Sigma^\infty S^V$. A priori $S^\alpha$ is only well defined up to non-canonical isomorphism, but in the context of computing Picard groups we allow ourselves this slight abuse of notation. Given an invertible $G$-spectrum $X$, each $\Phi^H(X)$ is a (non-equivariant) sphere, so $X$ determines a map $f$ from the set of conjugacy classes of subgroups of $G$ to the integers given by $\Phi^H(X) \simeq S^{f(H)}$. We prove that when $G=C_{p^n}$ and we work in the category of $p$-complete $G$-spectra, the map $f$ associated to $X$ determines $X$. This lets us recognize representation spheres.

In Section \ref{s:rec} we use the main result from the previous section to recognize certain more general spectra. Combined with the homology calculation in Section \ref{s:homology} this lets us recognize the stable homotopy type of $\Sigma(a,b,m)$ or $X(a,b,m)$ as a $C_{p^n}$-spectrum, where $n=\nu_p(m)$. And as explained above, that suffices to prove Theorem \ref{t:main}.

Finally, in Section \ref{s:integral} we prove Theorem \ref{t:mainZ}. The proof is similar to the proof of the main theorem in \cite{AnGeHe09}, only the combinatorics change because we have a different family of representations.

\subsection*{Relationship to other work}
After sharing a draft of this paper with Lars Hesselholt we learned that he and Thomas Nikolaus are in the process of writing up a calculation of the algebraic $K$-theory of $A$ in the case when $k$ is a perfect field of characteristic $p$ using a new approach to topological cyclic homology developed by Nikolaus and Scholze \cite{NiSc17}.

\subsection*{Acknowledgements}
I would like to thank Lars Hesselholt for telling me about this open problem in the first place. I would also like to thank Clover May for interesting conversations about equivariant stable homotopy theory.

\section{Topological Hochschild homology of $A$} \label{s:THH}
Following Hesselholt, we start by giving an explicit description of $THH(A, \mathfrak{a})$. Let $\Sigma(a,b,m)$ denote the sub-simplicial complex of $\Delta^{m-1}$ whose set of simplices consists of faces $[v_{r_0},\ldots,v_{r_e}]$ such that each $r_i-r_{i-1}$ as well as $r_0+m-r_e$ can be written as $ai+bj$ with $i, j \geq 0$. This is a $C_m$-equivariant space, where $C_m$ acts by cyclically permuting the set of vertices. This space is relevant because $A=k(\Pi)$ is the pointed monoid algebra on the pointed monoid $\Pi = \{0, 1, t^m \quad | \quad m=ai+bj\}$ with $i, j \geq 0$, so $THH(A) \simeq THH(k) \sma B^{cy}(\Pi)$. The space $B^{cy}(\Pi)$ breaks up as $B^{cy}(\Pi) = \bigvee_m B^{cy}(\Pi)[m]$, where the wedge summand labelled by $m$ is the geometric realization of the sub-simplicial set of $t$-degree exactly $m$ plus the disjoint basepoint $0$. For formal reasons we have $B^{cy}(\Pi)[m] \cong S^1_+ \sma_{C_m} \Sigma(a,b,m)_+$, so this leads us to try to understand the $C_m$-equivariant homotopy type of $\Sigma(a,b,m)$.

Let $X(a,b,m) = \Delta^{m-1}/\Sigma(a,b,m)$. (If $\Sigma(a,b,m) = \varnothing$ we interpret this as $\Delta^{m-1}_+$.) Since $\Delta^{m-1}$ is $C_m$-equivariantly contractible this is homotopy equivalent to the suspension of $\Sigma(a,b,m)$. Then Hesselholt proved Theorem \ref{t:main} under the additional assumption (Conjecture B) that $S^1_+ \sma_{C_m} X(a,b,m)$ has a certain $S^1$-equivariant homotopy type. We recall some of the details here, and the rest in Definition \ref{d:Ybeta} below. Pick $c, d \in \bZ$ such that the matrix $\begin{bmatrix} a & b \\ c & d \end{bmatrix}$ has determinant $1$. Let $\bC(\nu)$ denote the $S^1$-representation $\bC$ with $z \in S^1$ acting by multiplication by $z^\nu$, and use the same notation for the restriction to a $C_m$-representation. Define $\lambda(a,b,m) = \bigoplus\limits_{\nu \in (cm/a, dm/b) \cap \bZ} \bC(\nu)$. Then Hesselholt proved Theorem \ref{t:main} under the assumption that $S^1_+ \sma_{C_m} X(a,b,m)$ is $S^1$-equivariantly homotopy equivalent to a certain space $Y(a,b,m)$ built from the representation sphere $S^{\lambda(a,b,m)}$. For a precise statement, see Theorem \ref{t:main2} below and the definition of $Y(a,b,m)$ in Definition \ref{d:Ybeta} or on p.\ 3 of \cite{He14}.

The way Hesselholt tried to approach the conjecture about the homotopy type of the above space was by studying the stunted regular cyclic polytope $P(a,b,m)$ defined as the convex hull of
\[
 V(a,b,m) = \{ z^\nu \quad | \quad \nu \in J, z \in C_m\} \subset \bC^J,
\]
where $J = [cm/a, dm/b] \cap \bZ$. See \cite[Section 6]{He14} for details. This is similar, at least in spirit, to the proof strategy employed by Hesselholt and Madsen in \cite{HeMa97b} to study the $K$-theory of $k[x]/(x^a)$.

In fact, we now know how to produce a $C_m$-equivariant map $X(a,b,m) \to Y(a,b,m)$. The proof uses Fourier analysis. Unfortunately, we have been unable to prove that this map is an equivalence. Since stunted regular cyclic polytopes are interesting in their own right, we give the details in \cite{AnFourier}.

Instead, we will use methods from equivariant stable homotopy theory. We will always work with \emph{genuine} equivariant spectra. The point-set level model is not important, but for concreteness we choose to work in the category of orthogonal $G$-spectra. We are mostly interested in $G=C_{p^n}$, but parts of our paper apply to all finite groups.

We prove the following more technical result:

\begin{thm} \label{t:main2}
Given $(a,b,m)$ as above, let $n=\nu_p(m)$ and consider $X(a,b,m) = \Delta^{m-1}/\Sigma(a,b,m)$ as a $C_{p^n}$-space via the inclusion $C_{p^n} \subset C_m$. Then up to $p$-completion the suspension spectrum $\Sigma^\infty X(a,b,m)$ is $C_{p^n}$-equivariantly homotopy equivalent to $\Sigma^\infty Y(a,b,m)$.
\end{thm}
\noindent
Here $c$ and $d$ are as above, and $Y(a,b,m) = Y_{\lambda(a,b,m)}$ is defined in Definition \ref{d:Ybeta} below.

Given Theorem \ref{t:main2}, Theorem \ref{t:main} follows from \cite{He14} using Theorem \ref{t:main2} as a stand-in for \cite[Conjecture B]{He14}. The point is that Hesselholt's proof does not use the full force of Conjecture B: First, it suffices to consider the \emph{stable} $C_m$-equivariant homotopy type of $X(a,b,m)$ rather than its unstable $C_m$-equivariant homotopy type. Second, it suffices to consider the \emph{$C_{p^n}$-equivariant} stable homotopy type of $X(a,b,m)$ rather than its $C_m$-equivariant stable homotopy type, because we are interested in the $p$-typical version of topological cyclic homology only. And third, it suffices to consider the spectrum $\Sigma^\infty X(a,b,m)$ up to $p$-completion only because $THH(k)$ is $p$-complete.

\section{The homology of $\Sigma(a,b,m)$} \label{s:homology}
We will always let $x=t^a$ and $y=t^b$, so $k[x,y]/(x^b-y^a) = k \{ t^m \quad | \quad m=ai+bj\}$ with $i, j \geq 0$. In this section we will work over $\bZ$, so we define $A = \bZ[x,y]/(x^b-y^a)$. As in the previous section, let $\Sigma(a,b,m)$ be the sub-simplicial complex of $\Delta^{m-1}$ whose simplices consists of faces $[v_{r_0},\ldots,v_{r_e}]$ such that each $r_i-r_{i-1}$ as well as $r_0+m-r_e$ can be written as a sum of $a$'s and $b$'s. We write an $e$-simplex of $\Sigma(a,b,m)$ as
\[
 \sigma = t^{i-1} t_0 t^{k_0-i} | t^{k_1} | \ldots | t^{k_e},
\]
where $1 \leq i \leq k_0$, each $k_q$ can be written as a sum of $a$'s and $b$'s, and $\sum\limits k_q = m$. We recover the original description of the $e$-simplex by setting $k_0 = r_0 + m - r_e$, $i = m - r_e$ and $k_j = r_j - r_{j-1}$ for $j \geq 1$.

On the associated chain complex, the preferred generator $\tau$ of $C_m$ sends a non-degenerate simplex $\sigma$ to
\[
 \tau \sigma = \begin{cases} t^{i-2} t_0 t^{k_0-i+1} | t^{k_1} | \ldots | t^{k_e} \quad & \textnormal{if $i > 1$} \\ (-1)^e t^{k_e-1} t_0 | t^{k_0} | \ldots | t^{k_{e-1}}  \quad & \textnormal{if $i = 1$} \end{cases}
\]

As explained above, $B^{cy}(\Pi)[m] \simeq S^1_+ \sma_{C_m} \Sigma(a,b,m)_+$. But we can also model $\Sigma(a,b,m)_+$ directly using a version of the cyclic bar construction. To do this, we define $\wt{\Pi} = \{0, t^{i-1} t_0 t^{k-i} \quad | \quad t^k \in \Pi\}$. This is a pointed $\Pi$-biset, using the obvious maps $\Pi \times \wt{\Pi} \to \wt{\Pi}$ and $\wt{\Pi} \times \Pi \to \wt{\Pi}$ which multiply by powers of $t$ on the left or right. Then $\Sigma(a,b,m)_+ \cong B^{cy}(\Pi; \wt{\Pi})[m]$.

At the level of algebras and bimodules, we define an $A$-bimodule $\wt{A}$ by
\[
 \wt{A}_k = \begin{cases} \bigoplus\limits_{i=1}^k \bZ\{t^{i-1} t_0 t^{k-i}\} \quad & \textnormal{if $k$ is a sum of $a$'s and $b$'s} \\ 0 & \textnormal{otherwise} \end{cases}
\]
Then there is an isomorphism between the associated chain complex of $\Sigma(a,b,m)_+$ (which computes the reduced homology of $\Sigma(a,b,m)_+$) and the degree $m$ part of the chain complex $B^{cy}(A; \wt{A})$ computing $HH_*(A; \wt{A})$. This gives us a way to study the homology groups of $\Sigma(a,b,m)$ using homological algebra, and by working just a little bit harder we can also determine the $C_m$-action on $H_*(\Sigma(a,b,m))$.

\begin{remark}
The theory of cyclic sets does not give an $S^1$-action, or even a $C_m$-action, on $B^{cy}(A; \wt{A})[m]$. This means we have to work with explicit representatives to compute the action.
\end{remark}

\begin{thm} \label{t:homologyofsigma}
The homology groups of $\Sigma(a,b,m)$ as a $\bZ[C_m]$-module are given as follows. If $m$ can be written as a sum of $a$'s and $b$'s, so $\Sigma(a,b,m) \neq \varnothing$, then:
\begin{enumerate}
\item If $a \nmid m$ and $b \nmid m$ then $\wt{H}_q(\Sigma(a,b,m))$ is isomorphic to $\bZ[C_m/C_m]$ in dimension $q=2\ell(a,b,m)-1$ and zero otherwise.
\item If $a \mid m$ and $b \nmid m$ then $\wt{H}_q(\Sigma(a,b,m))$ is isomorphic to the kernel of
\[
 \bZ[C_m/C_{m/a}] \to \bZ[C_m/C_m]
\]
in dimension $q=2\ell(a,b,m)$ and zero otherwise.
\item If $a \nmid m$ and $b \mid m$ then $\wt{H}_q(\Sigma(a,b,m))$ is isomorphic to the kernel of
\[
 \bZ[C_m/C_{m/b}] \to \bZ[C_m/C_m]
\]
in dimension $q=2\ell(a,b,m)$ and zero otherwise.
\item If $a \mid m$ and $b \mid m$ then $\wt{H}_q(\Sigma(a,b,m))$ is isomorphic to the iterated kernel of the square
\[ \xymatrix{
 \bZ[C_m/C_{m/ab}] \ar[r] \ar[d] & \bZ[C_m/C_{m/a}] \ar[d] \\
 \bZ[C_m/C_{m/b}] \ar[r] & \bZ[C_m/C_m]
} \]
in dimension $q=2\ell(a,b,m)+1$ and zero otherwise.
\end{enumerate}
\end{thm}

The four cases can of course be given a unified description as the homology groups of the associated chain complex of the square in Case (4) starting in dimension $2\ell(a,b,m+a+b)-1$, with the convention that $\bZ[C_m/C_{m/d}] = 0$ if $d \nmid m$. The reduced homology groups of $X(a,b,m)$ are given by shifting the reduced homology groups of $\Sigma(a,b,m)$ up by one.

\begin{proof}
By the description of $H_*(\Sigma(a,b,m))$ as the degree $m$ part of the homology of $B^{cy}(A; \wt{A})$, we can build on the calculation given in \cite[Section 5]{He14}. In particular, we can start with the differential graded $A^e$-algebra
\[
 R(A) = A^e \otimes \Lambda(dx, dy) \otimes \Gamma(z),
\]
with $\delta(dx) = x \otimes 1 - 1 \otimes x$, $\delta(dy) = y \otimes 1 - 1 \otimes y$, and
\[
 \delta(z) = \sum_{u=1}^{b} x^{u-1} \otimes x^{b-u} \cdot dx - \sum_{v=1}^{a} y^{v-1} \otimes y^{a-v} \cdot dy.
\]
(We write $\delta$ for the differential, since $d$ is already used.) This is a projective resolution of $A$ as an $A^e$-module, so by tensoring over $A^e$ with $\wt{A}$ we get a chain complex
\[
 \wt{R}(A) = \wt{A} \otimes \Lambda(dx, dy) \otimes \Gamma(z)
\]
computing $HH_*(A, \wt{A})$. Here $\delta(\tilde{a} dx) = x\tilde{a} - \tilde{a}x$, $\delta(\tilde{a} dy) = y \tilde{a} - \tilde{a} y$, and
\[
 \delta(\tilde{a} z) = \sum_{u=1}^{b} x^{u-1} \tilde{a} x^{b-u} \cdot dx - \sum_{v=1}^{a} y^{v-1} \tilde{a} y^{a-v} \cdot dy. 
\]
Now we can compute $\wt{H}_*(\Sigma(a,b,m))$ in each of the cases in the theorem:

\begin{enumerate}
\item Write $t^m = x^i y^{j+ar} = x^{i+br} y^j$ with $0 < i < b$ and $0 < j < a$. Then $r=\ell(a,b,m)-1$. Define $\tilde{x} = \sum_{u=1}^a t^{u-1} t_0 t^{a-u}$ and $\tilde{y} = \sum_{v=1}^b t^{v-1} t_0 t^{b-v}$. These play the role of $ax$ and $by$ in the corresponding calculation of $HH_*(A)$. Then the homology of $\Sigma(a,b,m)$ in dimension $2\ell(a,b,m)-1$ is generated by
\[
 x^{i-1} y^{j-1} (\tilde{y} dx - \tilde{x} dy) z^{[r]}.
\]

To see that this element is a cycle, first suppose $r=0$ and $i=j=1$. Then
\begin{multline*}
 \delta( \tilde{y} dx - \tilde{x} dy ) = \sum_{v=1}^b t^{a+v-1} t_0 t^{b-v} - \sum_{v=1}^b t^{v-1} t_0 t^{a+b-v} \\
 - \sum_{u=1}^a t^{b+u-1} t_0 t^{a-u} + \sum_{u=1}^a t^{u-1} t_0 t^{a+b-u}.
\end{multline*}
Combining the first and fourth sums, and the second and third sums, gives $\sum\limits_{w=1}^{a+b} t^{w-1} t_0 t^{a+b-w} - \sum\limits_{w=1}^{a+b} t^{w-1} t_0 t^{a+b-w} = 0$. If $i$ and $j$ are arbitrary, this just adds a factor of $x^{i-1} y^{j-1}$ on the left everywhere.

If $r > 0$, the part of $\delta(x^{i-1} y^{j-1} (\tilde{y} dx - \tilde{x} dy) z^{[r]})$ with a factor of $z^{[r]}$ cancels as before, and the part with a factor of $z^{[r-1]}$ is given by $x^{i-1} y^{j-1}$ times
\begin{multline*}
 - \sum_{v'=1}^a \sum_{v=1}^b t^{b(v'-1)+v-1} t_0 t^{b(a-v')+b-v} dx dy z^{[r-1]} \\
 + \sum_{u'=1}^b \sum_{u=1}^a t^{a(u'-1)+u-1} t_0 t^{a(b-u')+a-u} dx dy z^{[r-1]}.
\end{multline*}
Again we see that the sums cancel since they are both given by
\[
 \sum\limits_{w=1}^{ab} t^{w-1} t_0 t^{ab-w} dx dy z^{[r-1]}.
\]

To see that no multiple of this class is a boundary, note that if $i=j=1$ then there are no classes in homological degree $2r+2$ and $t$-degree $m$. If $i > 1$ or $j > 1$ then there are some such classes, given by $t^w t_0 t^{w'} dx dy z^{[r]}$, but $\delta$ applied to such a class with $w'$ maximal has a term $t^w t_0 t^{w'+b} dx z^{[r]}$ which is not a term of our cycle and cannot cancel with anything else.

To see that there is no more homology in dimension $2\ell(a,b,m)-1$, note that any cycle
\[
 \sum \alpha_w t^w t_0 t^{w'} dx z^{[r]} + \sum \beta_w t^w t_0 t^{w'} dy z^{[r]} + \textnormal{terms with $z^{[r']}$ for $r'<r$}
\]
with $w+w'=a(i-1)+bj-1$ in the first sum and $w+w'=ai+b(j-1)-1$ in the second sum, is homologous to such a sum where $w' < b$ in the first sum. If the second sum has any terms with $w' \geq a$, $\delta$ applied to such a class with $w'$ maximal has a term which cannot cancel with anything else so this class cannot be a cycle. This is the key reduction; we omit the rest of the details.

\item Write $t^m = x^i y^{ar} = x^{i+br}$ with $0 < i < b$. Then $r=\ell(a,b,m)$. For each $1 \leq q \leq a-1$ we then have a generator
\[
 x^{i-1} t^{q-1} (t_0 t - t t_0) t^{a-1-q} z^{[r]} + x^{i-1} t^{q-1} \tilde{w} t^{a-1-q} dx dy z^{[r-1]},
\]
where
\[
 \tilde{w} = \sum t^e t_0 t^{(a-1)(b-1)-e} - \sum t^{e+1} t_0 t^{(a-1)(b-1)-(e+1)},
\]
and the two sums are over all $e$ that can be written as a sum of $a$'s and $b$'s, with $0 \leq e \leq (a-1)(b-1)$ in the first sum and $0 \leq e \leq (a-1)(b-1)-1$ in the second sum. If $r=0$, omit the terms with $z^{[r-1]}$.

To see that this is a cycle, we calculate its boundary. We omit the factors $x^{i-1}t^{q-1}$ and $t^{a-1-q}$ as the play no role in the calculation. Then the boundary of the first term is given by
\begin{multline*}
 \Big( \sum t^e t_0 t^{(a-1)(b-1)+b-e} - \sum t^{e+1} t_0 t^{(a-1)(b-1)+b-(e+1)} \Big) dx z^{[r-1]} \\ 
 - \Big( \sum t^e t_0 t^{(a-1)(b-1)+a-e} - \sum t^{e+1} t_0 t^{(a-1)(b-1)+a-(e+1)} \Big) dy z^{[r-1]},
\end{multline*}
where the first two sums are over those $e$ which are a multiple of $a$ and the last two sums are over those $e$ which are a multiple of $b$. The boundary of the second term is
\begin{multline*}
 \Big( -\sum t^e t_0 t^{(a-1)(b-1)+b-e} + \sum t^{e+1} t_0 t^{(a-1)(b-1)+b-(e+1)} \Big) dx z^{[r-1]} \\
 - \Big( -\sum t^{e+b} t_0 t^{(a-1)(b-1)-e} + \sum t^{e+b+1} t_0 t^{(a-1)(b-1)-(e+1)} \Big) dx z^{[r-1]} \\
 + \Big( \sum t^e t_0 t^{(a-1)(b-1)+a-e} - \sum t^{e+1} t_0 t^{(a-1)(b-1)+a-(e+1)} \Big) dy z^{[r-1]} \\
 - \Big( \sum t^{e+a} t_0 t^{(a-1)(b-1)-e} - \sum t^{e+a+1} t_0 t^{(a-1)(b-1)-(e+1)} \Big) dy z^{[r-1]}
\end{multline*}
where the sums are as in the definition of $\tilde{w}$. Now a combinatorial argument shows that everything cancels. For example, consider the terms with a factor of $dx z^{[r-1]}$. We have a total of $6$ sums with a factor of $dx z^{[r-1]}$, with $2$ coming from the boundary of the first term and $4$ from the boundary of the second term. Then consider all the terms starting with $t^e t_0$.
\begin{itemize}
\item If $e \equiv 0 \mod a$, the terms in sums $1$ and $3$ cancel for $e < a(b-2)$, the terms in sums $4$ and $6$ cancel for $e < a(b-2)$, and the terms in sums $1$ and $6$ cancel for $e \geq a(b-2)$.
\item If $e \equiv 1 \mod a$, the terms in sums $2$ and $4$ cancel for $e < a(b-2)$, the terms in sums $3$ and $5$ cancel for $e < a(b-2)$, and the terms in sums $2$ and $5$ cancel for $e \geq a(b-2)$.
\item If $e \neq 0, 1 \mod a$, the terms in sums $3$ and $5$ cancel for $e < a(b-2)$, the terms in sums $4$ and $6$ cancel for $e < a(b-2)$, and the terms in sums $5$ and $6$ cancel for $e \geq a(b-2)$.
\end{itemize}

The proof that the terms with a factor of $dy z^{[r-1]}$ cancel is identical. The proof that none of these classes are boundaries, and that this accounts for all the homology, is similar to the previous part.

\item Write $t^m = x^{br} y^j = y^{j+ar}$ with $0 < j < a$. Then $r=\ell(a,b,m)$. For each $1 \leq q \leq b-1$ we then have a generator
\[
 y^{j-1} t^{q-1}(t_0 t - t t_0) t^{b-1-q} z^{[r]} + y^{j-1} t^{q-1} \tilde{w} t^{b-1-q} dx dy z^{[r-1]},
\]
where $\tilde{w}$ is as above. If $r=0$, omit the terms with $z^{[r-1]}$. The proof is identical to the previous part.

\item Let $r=\ell(a,b,m)$. For each $1 \leq q \leq (a-1)(b-1)$ we then have a generator
\[
 t^{q-1} \big( \tilde{y}dx - \tilde{x} dy \big) t^{(a-1)(b-1)-q} z^{[r]}.
\]
The proof that these classes generate the homology is similar to the previous parts.
\end{enumerate}

Understanding the $C_m$-action takes a little bit more work, because the model $\wt{A} \otimes \Lambda(dx,dy) \otimes \Gamma(z)$ does not give a transparent way of seeing the action. Instead we use that in $B^{cy}(A; \wt{A})$, $dx$ is represented by $1 \otimes x$, $dy$ is represented by $1 \otimes y$, and $z$ is represented by
\[
 \sum_{u=1}^{b-1} x^{u-1} \otimes x^{b-u} \otimes x - \sum_{v=1}^{a-1} y^{v-1} \otimes y^{a-v} \otimes y.
\]
By using the shuffle product we can then get explicit representatives for all the homology generators in $B^{cy}(A; \wt{A})$.

Now, to compute the action of the preferred generator $\tau$ of $C_m$ on one of the generators of $\wt{H}_q(\Sigma(a,b,m))$ it suffices to find a representative for the generator in $B^{cy}(A;\wt{A})$, compute the action by $\tau$, and express the result as a linear combination of the generators.

We give some of the details in Case (4), the other cases are similar but easier. As an abelian group $\bZ[C_m/C_{m/ab}]$ has rank $(a-1)(b-1)$, and we can choose generators $e_1,\ldots,e_{(a-1)(b-1)}$ in such a way that $\tau \cdot e_i = e_{i+1}$ for $i<(a-1)(b-1)$ and $\tau \cdot e_{(a-1)(b-1)}$ is given by
\[
 \tau \cdot e_{(a-1)(b-1)} = \sum \big( -e_{w+1} + e_{w+2} \big),
\]
where the sum is over all $w \in [0, (a-1)(b-1)-2]$ such that $w$ is a sum of $a$'s and $b$'s.

In the case at hand, let $e_i$ be the generator corresponding to $q=(a-1)(b-1)+1-i$. Then it is clear that for $i \leq (a-1)(b-1)-1$ we have $\tau \cdot e_i = e_{i+1}$. To calculate $\tau \cdot e_{(a-1)(b-1)}$, first consider the case $r=0$. Then
\begin{multline*}
 e_{(a-1)(b-1)} = \tilde{y} t^{(a-1)(b-1)-1} dx - \tilde{x} t^{(a-1)(b-1)-1} dy = \\ 
 \sum_{v=1}^b t^{v-1} t_0 t^{(a-1)(b-1)-1+b-v} \otimes x 
 - \sum_{u=1}^a t^{u-1} t_0 t^{(a-1)(b-1)-1+a-u} \otimes y.
\end{multline*}
If we apply $\tau$ we get
\begin{multline*}
 \tau \cdot e_{(a-1)(b-1)} = \sum_{v=2}^b t^{v-2} t_0 t^{(a-1)(b-1)+b-v} \otimes x - \sum_{u=2}^a t^{u-2} t_0 t^{(a-1)(b-1)+a-u} \otimes y \\
 - t^{a-1} t_0 \otimes x^{b-1} + t^{b-1} t_0 \otimes y^{a-1}.
\end{multline*}
This is not a linear combination of our generators, but it is homologous to one. If we add
\[
 \sum_{u=1}^{b-1} -d(x^{u-1} t^{a-1} t_0 \otimes x^{b-u} \otimes x) + \sum_{v=1}^{a-1} d(y^{v-1} t^{b-1} t_0 \otimes y^{a-v} \otimes y),
\]
we get precisely $\tau \cdot e_{(a-1)(b-1)}$, so that finishes the proof for $r=0$. (We write $d$ for the differential in the cyclic bar resolution rather than $\delta$.)

When $r > 0$, the same procedure works except we have to adjust $\tau \cdot e_{(a-1)(b-1)}$ by $d$ applied to
\[
 \Big( \sum_{u=1}^{b-1} -x^{u-1} t^{a-1} t_0 \otimes x^{b-u} \otimes x + \sum_{v=1}^{a-1} y^{v-1} t^{b-1} t_0 \otimes y^{a-v} \otimes y \Big) \cdot z^{[r]}
\]
instead. (This involves using the shuffle product.) We omit the details as they are not particularly enlightening.
\end{proof}

This shows that the homology groups of $X(a,b,m)$ agrees with the homology groups of $Y_{\lambda(a,b,m)}$, giving strong circumstantial evidence that $X(a,b,m)$ and $Y(a,b,m)$ are $C_m$-equivariantly homotopy equivalent.

We can also give an alternative proof of the main theorem in \cite{HeMa97b}. The homological algebra input is as follows. Following \cite[Definition 3.2]{An14}, let $X_{s,a}$ denote the pointed simplicial set generated by an $(s-1)$-simplex $x_0 | x | \ldots | x$ ($s$ factors) and a basepoint $*$, with face and degeneracy maps as in the Hochschild chain complex and with relations $x^a=*$ and $x^{i-1} x_0 x^{a-i} = *$. The $C_s$-action is similar to the $C_m$-action defined on $\Sigma(a,b,m)$ above. Let $B = \bZ[x]/(x^a)$ and define a $B$-bimodule $\wt{B}$ by $\wt{B}_k = \bigoplus\limits_{i=1}^{k-1} \bZ\{x^{i-1} x_0 x^{k-i}\}$ for $1 \leq k \leq a-1$. Then there is an isomorphism between the simplicial chain complex of $X_{s,a}$ and the degree $s$ part of the chain complex $B^{cy}(B; \wt{B})$ computing $HH_*(B; \wt{B})$.

By an elaboration of \cite[Proposition 3.4]{An14} we can then prove the following:

\begin{thm} \label{t:homologyoftruncated}
The reduced homology of $X_{s,a}$ as a $\bZ[C_s]$-module is given as follows:
\begin{enumerate}
 \item If $a \nmid s$ then $\wt{H}_q(X_{s,a})$ is isomorphic to $\bZ[C_s/C_s]$ in dimension $2d$ and zero otherwise. 
 \item If $a \mid s$ then $\wt{H}_q(X_{s,a})$ is isomorphic to the kernel of $\bZ[C_s/C_{s/a}] \to \bZ[C_s/C_s]$ in dimension $2d+1$ and zero otherwise.
\end{enumerate}
Here $d = \lfloor (s-1)/a \rfloor$.
\end{thm}

The proof uses the explicit representatives for the homology groups given in \cite[Proposition 3.4]{An14}. Since the details are similar to the proof of Theorem \ref{t:homologyofsigma} and this result will only be used to reprove a known result, we omit the details of the proof.

\section{The Picard group of $(Sp^{C_{p^n}})^\wedge_p$} \label{s:picard}
Recall that the Picard group of a symmetric monoidal category is the group of isomorphism classes of invertible objects. We are interested in the Picard group of the homotopy category of genuine $G$-spectra for a finite group $G$. Later we will specialize to $G=C_{p^n}$ but for now we can let $G$ be any finite group. Then we have a map $RO(G) \to \pic(Sp^G)$ given by $\alpha \mapsto S^\alpha$. As explained in the introduction, if $\alpha = [V] - [W]$ then $S^\alpha = \Sigma^{-W} \Sigma^\infty S^V$. A priori this is only well defined modulo non-canonical isomorphism, but for the purpose of computing Picard groups that does not matter.

The Picard group is the appropriate indexing set for homotopy groups. If $\cC$ is some symmetric monoidal category and $[X] \in \pic(\cC)$ then we are interested in the functor $\pi_X(-) = [X, -]$, and if $[X] = [X']$ in $\pic(\cC)$ then obviously $\pi_X(-)$ and $\pi_{X'}(-)$ are isomorphic functors. In particular, if $\alpha, \alpha' \in RO(G)$ are such that $S^\alpha \simeq S^{\alpha'}$ then the functors $\pi_\alpha^G(-)$ and $\pi_{\alpha'}^G(-)$ are going to be isomorphic. It sometimes happens that $\alpha \neq \alpha'$ in $RO(G)$ but $S^\alpha \simeq S^{\alpha'}$, which means that the full representation ring is not the natural indexing set for homotopy groups.

There are incarnations of this in the literature. For example, in \cite[Theorem 2.16]{Ze17}, which Zeng attributes to Hu and Kriz, Zeng shows the following: Let $G=C_{p^n}$, and suppose $\gamma$ and $\gamma'$ are primite $p^i$'th roots of unity. Then $S^\gamma \sma H\underline{\bZ} \simeq S^{\gamma'} \sma H\underline{\bZ}$. That allows for a similar conclusion as in Theorem \ref{t:Picard}, but after smashing with $H\underline{\bZ}$ instead of completing at $p$, and lets Zeng index $H\underline{\bZ}$-(co)homology on a smaller indexing set than $RO(G)$.

For another example, consider the $RO(C_{p^n})$-graded equivariant homotopy groups of $THH(\bF_p)^{C_{p^n}}$ or, with mod $p$ coefficients, of $THH(\bZ)^{C_{p^n}}$. These have been determined by Gerhardt \cite{Ge08} and the author and Gerhardt \cite{AnGe11}. In both cases, the calculations could not distinguish between primitive $p^i$'th roots of unity, but no conceptual reason for this was given.

Let $C(G)$ denote the ring of functions from the set of conjugacy classes of subgroups of $G$ to $\bZ$, considered as an abelian group. Then we have a map $\Psi : \pic(Sp^G) \to C(G)$ given by $\Psi(X) = f$ where $f$ is defined by $\Phi^H(X) \simeq S^{f(H)}$. For example, if $\alpha = [V] - [W]$ in $RO(G)$ then $\Psi(S^\alpha)(H) = \dim(V^H) - \dim(W^H)$. By the main result of \cite{FLM01}, there is an exact sequence
\[
 0 \to \pic(A(G)) \to \pic(Sp^G) \xto{\Psi} C(G)
\]
Here $A(G)$ is the Burnside ring of $G$, defined as the group completion of the monoid of isomorphism classes of finite $G$-sets under disjoint union. The ring structure comes from Cartesian product. As a group, $A(G)$ is free abelian with a basis consisting of $G/H$ where $H$ runs through the set of conjugacy classes of subgroups of $G$.

The map $\Psi : \pic(Sp^G) \to C(G)$ is not in general surjective: For example, if $G=C_p$ for an odd prime $p$ then a function $f \in C(G)$ is in the image if and only if $f(e) \equiv f(G) \mod 2$. And the composite $\Psi' : RO(G) \to \pic(Sp^G) \to C(G)$ is not usually injective. For $G=C_{p^n}$, $\Psi'$ is injective only for $p^n=2, 3, 4$. For example, when $G=C_5$ the spectrum $S^{\bC(1)-\bC(2)}$ is in the kernel of $\Psi'$. But $S^{\bC(1)-\bC(2)} \sma S^{\bC(1)-\bC(2)} \simeq S^0$, so the corresponding element of $\pic(A(C_5))$ has order $2$ and in fact $\pic(A(C_5)) \cong \bZ/2$.

The group $\pic(A(G))$ is finite, so when $\Psi'$ has a non-trivial kernel the map $RO(G) \to \pic(Sp^G)$ has a non-trivial kernel as well. In particular, when $\Psi'$ has a non-trivial kernel there are $\alpha \neq \alpha'$ in $RO(G)$ with $S^\alpha \simeq S^{\alpha'}$. For example, when $G=C_5$ that is true for $\alpha=2\bC(1)$ and $\alpha'=2\bC(2)$.

We are interested in the category of $p$-complete $G=C_{p^n}$-spectra for a prime $p$. We will use the notation $X_{hG}$, $X^{hG}$ and $X^{tG}$ for the homotopy orbit spectrum, the homotopy fixed point spectrum, and the Tate spectrum of $G$ acting on $X$, respectively.

\begin{thm} \label{t:Picard}
Let $G=C_{p^n}$. Then map $\Psi : \pic((Sp^G)^\wedge_p) \to C(G)$ given by $\Psi(X) = f$ where $f$ is defined by $\Phi^H(X) \simeq (S^{f(H)})^\wedge_p$ is injective.
\end{thm}

\begin{proof}
Suppose $X$ is in the kernel of $\Psi : \pic((Sp^G)^\wedge_p) \to C(G)$, so $\Phi^H(X) \simeq (S^0)^\wedge_p$ for each $H \leq G$. Then it suffices to show that $X \simeq (S^0_G)^\wedge_p$. To unclutter the notation, let us implicitly $p$-complete all spectra in the rest of the proof.

Let $X_i = \Phi^{C_{p^{n-i}}}(X)$. Then we will prove by induction on $i$ that $X_i \simeq S^0_{C_{p^i}}$. Here $C_{p^i} = C_{p^n}/C_{p^{n-i}}$, which acts on $X_i$. The base case $i=0$ is immediate. Now consider the following diagram, which we will refer to as the fundamental diagram:

\[ \xymatrix{
 (X_i)_{hC_{p^i}} \ar[r] \ar[d]^= & X_i^{C_{p^i}} \ar[r] \ar[d]^\Gamma & X_{i-1}^{C_{p^{i-1}}} \ar[d]^{\wh{\Gamma}} \\
 (X_i)_{hC_{p^i}} \ar[r] & X_i^{hC_{p^i}} \ar[r] & X_i^{tC_{p^i}}
} \]
Here the rows are cofiber sequence.

By induction we can assume that $X_{i-1} \simeq S^0_{C_{p^{i-1}}}$, detected by a $C_{p^{i-1}}$-equivariant map $f_{i-1} : S^0_{C_{p^{i-1}}} \to X_{i-1}$. From \cite[Lemma 9.1]{HeMa97} it follows that the Tate spectrum $X_i^{tC_{p^i}}$ only depends on $X_{i-1}$, so by the Segal conjecture $\wh{\Gamma}$ is an equivalence. (To see that loc.\ cit.\ applies, use that any finite $G$-CW spectrum $X$ is given by $S^{-N\rho} \sma W$ for a finite $G$-CW space $W$ for a sufficiently large $N$, where $\rho$ is the regular representation.)

Now we claim that $\wh{\Gamma}(f_{i-1})$ is represented by a unit times the generator of $\wh{E}_2^{0,0}(X_i) \cong \bZ/p^i$ in the Tate spectral sequence converging to $\pi_* X_i^{tC_{p^i}}$. This requires some justification, as $\pi_0 X_i^{tC_{p^i}} \cong \bZ_p^i$ and $\wh{\Gamma}(f_{i-1})$ could a priori be represented in higher filtration in the Tate spectral sequence.

To see that $\wh{\Gamma}(f_{i-1})$ does indeed map to a unit times the generator of $\wh{E}_2^{0,0}(X_i)$ we compare the fundamental diagram for $X$ to the fundamental diagram for $X' = X \sma THH(\bF_p)$. So let $X_i' = \Phi^{C_{p^{n-i}}}(X \sma THH(\bF_p))$. We know that $f_{i-1}$ maps to a generator $f_{i-1}'$ of $\pi_0 (X'_{i-1})^{C_{p^{i-1}}} \cong \bZ/p^i$, and that $\wh{\Gamma} : (X'_{i-1})^{C_{p^{i-1}}} \to (X'_i)^{tC_{p^i}}$ is an equivalence on connective covers and in particular an isomorphism on $\pi_0$. Moreover, we can describe the Tate spectral sequence explicitly. The $E_2$-term of the Tate spectral sequence convering to $\pi_* (X_i')^{tC_{p^i}}$ is isomorphic to the $E_2$-term of the Tate spectral sequence converging to $\pi_* THH(\bF_p)^{tC_{p^i}}$, and since our Tate spectral sequence is a module over the one converging to $\pi_* THH(\bF_p)^{tC_{p^i}}$, this forces the differentials in the two spectral sequences to be isomorphic. It follows that any unit in $\pi_0 (X_{i-1}')^{C_{p^{i-1}}}$ maps to an element in $\pi_0 (X_i')^{tC_{p^i}}$ represented by a unit in $\wh{E}_2^{0,0}(X_i') \cong \bZ/p$. Now a diagram chase and naturality with respect to the map $X \to X'$ induced by $S \simeq THH(S) \to THH(\bF_p)$ shows that $\wh{\Gamma}(f_{i-1})$ is represented by a unit in $\wh{E}_2^{0,0}(X_i) \cong \bZ/p^i$.

By the top cofiber sequence in the fundamental diagram it follows that $\pi_0 X_i^{C_{p^i}} \to \pi_0 X_{i-1}^{C_{p^{i-1}}}$ is split surjective, so we can choose a lift $f_i$ of $f_{i-1}$. The commutativity of the right hand side square then shows that $\Gamma(f_i) \in \pi_0 X_i^{hC_{p^i}}$ is represented by a unit in $E_2^{0,0}(X_i) = \bZ_p$, where $E_2^{*,*}(X_i)$ is the $E_2$-term of the homotopy fixed point spectral sequence converging to $\pi_* X_i^{hC_{p^i}}$. It follows that $f_i$, thought of as a $C_{p^i}$-equivariant map $S^0_{C_{p^i}} \to X_i$, is an equivalence on underlying spectra. By construction $f_i$ is also an equivalence on geometric fixed points, so $f_i$ an equivalence and $X_i \simeq S^0_{C_{p^i}}$. This completes the induction step.
\end{proof}

This also lets us recognize representation spheres:

\begin{cor} \label{c:recrepsphere}
Let $G=C_{p^n}$. Suppose $X$ is a bounded below $G$-spectrum of finite type, and suppose there is some $\beta \in RO(G)$ such that the homology groups of $\Phi^{C_{p^{n-i}}}(X)$ agree with the homology groups of $\Phi^{C_{p^{n-i}}}(S^\beta)$ as $\bZ[C_{p^i}]$-modules for each $0 \leq i \leq n$. Then $X^\wedge_p \simeq (S^\beta)^\wedge_p$.
\end{cor}

\begin{proof}
Let $Z = F(S^\beta, X)^\wedge_p$. For each $0 \leq i \leq n$ we use the non-equivariant Hurewicz theorem to conclude that $\Phi^{C_{p^{n-i}}}(Z) \simeq (S^0)^\wedge_p$. Then the proof of Theorem \ref{t:Picard} goes through and shows that $Z \simeq (S^0_{C_{p^n}})^\wedge_p$, which gives the result.
\end{proof}

\begin{remark}
We are not making any statement about whether or not the above map $RO(C_{p^n}) \to \pic((Sp^{C_{p^n}})^\wedge_p)$ is surjective. It is possible (though we think it is unlikely) that there are exotic $p$-complete $C_{p^n}$-spheres that do not come from representation spheres.
\end{remark}

\section{Recognizing $G$-spectra} \label{s:rec}
In this section we leverage the calculation in Theorem \ref{t:Picard} above to recognize more general $G$-spectra.

\begin{defn} \label{d:Ybeta}
Let $\beta \in RO(C_m)$. Given relatively prime natural numbers $a$ and $b$, define a $C_m$-spectrum $Y_\beta$ as follows:
\begin{enumerate}
\item If $a \nmid m$ and $b \nmid m$, let $Y_\beta = (C_m/C_m)_+ \sma S^\beta$.
\item If $a \mid m$ and $b \nmid m$, let $Y_\beta$ be the homotopy cofiber of the map
\[
 (C_m/C_{m/a})_+ \sma S^\beta \to (C_m/C_m)_+ \sma S^\beta.
\]
\item If $a \nmid m$ and $b \mid m$, let $Y_\beta$ be the homotopy cofiber of the map
\[
 (C_m/C_{m/b})_+ \sma S^\beta \to (C_m/C_m)_+ \sma S^\beta.
\]
\item If $a \mid m$ and $b \mid m$, let $Y_\beta$ be the iterated homotopy cofiber of the diagram
\[ \xymatrix{
 (C_m/C_{m/ab})_+ \sma S^\beta \ar[r] \ar[d] & (C_m/C_{m/a})_+ \sma S^\beta \ar[d] \\
 (C_m/C_{m/b})_+ \sma S^\beta \ar[r] & (C_m/C_m)_+ \sma S^\beta 
} \]

\end{enumerate}

\end{defn}

\begin{thm} \label{t:recognizing}
Let $a$ and $b$ be relatively prime natural numbers. Suppose $X$ is a bounded below $C_m$-spectrum of finite type, let $n=\nu_p(m)$, and consider $X$ as a $C_{p^n}$-spectrum via the inclusion $C_{p^n} \to C_m$. If $\beta \in RO(C_m)$ is such that the homology groups of $\Phi^{C_{p^{n-i}}}(X)$ agree with the homology groups of $\Phi^{C_{p^{n-i}}}(Y_\beta)$ as $\bZ[C_{p^i}]$-modules for each $0 \leq i \leq n$ then $X^\wedge_p \simeq (Y_\beta)^\wedge_p$ as $C_{p^n}$-spectra.
\end{thm}

\begin{proof}
We can assume without loss of generality that $p \nmid a$. The case $a \nmid m$, $b \nmid m$ is precisely Corollary \ref{c:recrepsphere} above. Since $p \nmid a$, going from the case $a \nmid m$ to the case $a \mid m$ only introduces an extra $\bigvee_{a-1} \Sigma(-)$ everywhere which plays no role in the argument. Similarly, if $p \nmid b$ then the result follows from Corollary \ref{c:recrepsphere}. So it suffices to prove the result in the case $a \nmid m$, $b \mid m$, under the additional assumption that $p \mid b$.

Let $Z=F(Y_\beta,X)$. Then we wish to find an element $f \in \pi_0 (Z^{C_{p^n}})^\wedge_p$ which is adjoint to a $C_{p^n}$-equivariant homotopy equivalence $(Y_\beta)^\wedge_p \to X^\wedge_p$. Since $Y_\beta$ is a finite $G$-CW complex,
\[
 \Phi^{C_{p^{n-i}}}(Z) \simeq F(\Phi^{C_{p^{n-i}}}(Y_\beta), \Phi^{C_{p^{n-i}}}(X))
\]
for each $0 \leq i \leq n$. For ease of notation, let $Z_i = \Phi^{C_{p^{n-i}}}(Z)$. We proceed by induction on $i$. For $i < \nu_p(b)$, the argument is identical to the argument in the proof of Theorem \ref{t:Picard}, since
\[
 \Phi^{C_{p^{n-i}}} \cof\big(\Sigma^\infty (C_m/C_{m/b})_+ \to \Sigma^\infty (C_m/C_m)_+\big) \simeq S^0_{C_{p^i}}.
\]

Now assume we have $f_{i-1} \in \pi_0 Z_{i-1}^{C_{p^{i-1}}}$ adjoint to an equivalence $\Phi^{C_{p^{n-i+1}}}(Y_\beta)^\wedge_p \to \Phi^{C_{p^{n-i+1}}}(X)^\wedge_p$. For $i \geq \nu_p(b)$, we consider once again the fundamental diagram. Let $A = \ker(\bZ[C_m/C_{m/b}] \to \bZ[C_m/C_m])$ considered as a $C_{p^i} = C_{p^n}/C_{p^{n-i}}$-module, and note that $Hom(A,A) \cong \bZ[C_m/C_m] \oplus \bigoplus\limits_{b-2} \bZ[C_m/C_{m/b}]$. From the non-equivariant Hurewicz theorem it follows that $Z_i$ is a wedge sum of $(b-1)^2$ copies of $S^0$, and that up to homotopy $Z_i$ looks like $\big((C_m/C_m)_+ \sma S^0\big) \vee \big( \bigvee\limits_{b-2} (C_m/C_{m/b})_+ \sma S^0 \big)$. As a $C_{p^i}$-spectrum this looks up to homotopy like $S^0 \vee \bigvee\limits_{b(b-2)/p^{\nu_p(b)}} (C_{p^i}/C_{p^{i-\nu_p(b)}})_+ \sma S^0$.

From the Segal conjecture combined with the fact that $Z_i^{tC_{p^i}}$ only depends on $Z_{i-1}$ we conclude that $\wh{\Gamma} : Z_{i-1}^{C_{p^{i-1}}} \to Z_i^{tC_{p^i}}$ is a $p$-completion.

We claim that $\wh{\Gamma}(f_{i-1})$ is represented by an element that projects onto a unit times the generator of the $\bZ/p^i$ in $\wh{E}_2^{0,0}(Z_i)$ coming from the $S^0$ wedge summand. For this we again consider $Z' = Z \sma THH(\bF_p)$ and $Z'_i = \Phi^{C_{p^{n-i}}}(Z')$. Then
\[
 \pi_0 (Z')^{tC_{p^i}} \cong \bZ/p^i \oplus \bigoplus_{b(b-2)/\nu_p(b)} \bZ/p^{i-\nu_p(b)},
\]
with the generator of the $\bZ/p^i$ represented on $\wh{E}_2^{0,0}$. Since $\wh{\Gamma}$ is an isomorphism on $\pi_0$ in this case, it follows that $\wh{\Gamma}$ applied to a generator of the $\bZ/p^i$ in $\pi_0 (Z'_{i-1})^{C_{p^{i-1}}}$ must project onto a unit times the corresponding generator in $\wh{E}_2^{0,0}(Z_i')$. Now a diagram chase and naturality with respect to $Z \to Z'$ shows that $\wh{\Gamma}(f_{i-1})$ is represented by an element that projects onto a unit times the generator of the $\bZ/p^i$ in $\wh{E}_2^{0,0}(Z_i)$.

The map $\pi_0 Z_i^{C_{p^i}} \to \pi_0 Z_{i-1}^{C_{p^{i-1}}}$ is split surjective, so we can choose a lift $f_i \in \pi_0 Z_i^{C_{p^i}}$ of $f_{i-1}$. It follows that $\Gamma(f_i)$ is represented in $E_2^{0,0}(Z_i)$, the $E_2$-term of the homotopy fixed point spectral sequence, by a lift of $\wh{\Gamma}(f_{i-1})$. And any lift of $\wh{\Gamma}(f_{i-1})$ to $E_2^{0,0}$ will surject onto a $p$-adic unit in the summand corresponding to $id : A \to A$. Any other summand in $E_2^{0,0}(Z_i)$ corresponds to a map $A \to A$ whose matrix has determinant divisible by $p$. Hence $\Gamma(f_i)$ maps via the edge homomorphism in the homotopy fixed point spectral sequence to (an element adjoint to) a map $\Phi^{C_{p^{n-i}}}(Y_\beta) \to \Phi^{C_{p^{n-i}}}(X)$ which is an equivalence after $p$-completion. This finishes the induction step.
\end{proof}

By combining Theorem \ref{t:homologyofsigma} and Theorem \ref{t:recognizing}, we conclude that the stable homotopy type of $X(a,b,m)$, as a $C_{p^n}$-spectrum with $n=\nu_p(m)$, is given by the suspension spectrum of the space $Y(a,b,m)$ defined on p.\ 3 of \cite{He14}. This proves Theorem \ref{t:main2}, and as explained in the introduction that suffices to prove Theorem \ref{t:main}.

Similarly, by combining Theorem \ref{t:homologyoftruncated} and Theorem \ref{t:recognizing}, we conclude that the stable homotopy type of $X_{s,a}$, as a $C_{p^n}$-spectrum with $n=\nu_p(s)$, is given by the suspension spectrum of the space defined just after \cite[Theorem B]{HeMa97b}. That gives an alternative proof of the calculation of $K(k[x]/(x^a), (x))$ in \cite[Theorem A]{HeMa99}.

\section{Proof of Theorem \ref{t:mainZ}} \label{s:integral}
In this section we prove Theorem \ref{t:mainZ}.

\begin{prop} \label{p:negcyc}
Rationally the negative cyclic homology of $A=\bZ[x,y]/(x^b-y^a)$ relative to $\mathfrak{a}$ is given by $\bQ^{(a-1)(b-1)}$ in each positive odd degree and zero otherwise.
\end{prop}

\begin{proof}
This follows from \cite[Proposition 5.1]{He14} and the standard spectral sequence computing negative cyclic homology.
\end{proof}

\begin{proof}[Proof of Theorem \ref{t:mainZ}]
We will show that for each prime $p$, $\TC_q(A, \mathfrak{a}; p)$ is given by $\bZ_p^{(a-1)(b-1)/2}$ for $q=0$, by $\bZ_p^{(a-1)(b-1)}$ for $q=2i > 0$ even, and by a finite $p$-group for $q=2i-1 > 0$ odd. Together with Proposition \ref{p:negcyc} this implies that $K_q(A, \mathfrak{a})$ is finitely generated, given by $\bZ^{(a-1)(b-1)/2}$ for $q=0$, $\bZ^{(a-1)(b-1)}$ for $q=2i > 0$ even, and by the product of the finite $p$-group $\TC_q(A, \mathfrak{a}; p)$ over all $p$ for $q=2i-1 > 0$ odd.

By the same argument as in \cite{He14}, see in particular \cite[Diagram 4.3]{He14}, it follows that to understand $\TC_*(A, \mathfrak{a}; p)$ it suffices to understand the diagram
\[ \xymatrix{
 \lim\limits_R \bigoplus\limits_{a, b \mid m} \TR^{\nu_p(m)-\nu_p(a)-\nu_p(b)}_{q-\lambda(a,b,m)}(\bZ; p)_{(p)} \ar[r]^-{V^{\nu_p(a)}} \ar[d]_{V^{\nu_p(b)}} &  \lim\limits_R \bigoplus\limits_{b \mid m} \TR^{\nu_p(m)-\nu_p(b)}_{q-\lambda(a,b,m)}(\bZ; p)_{(p)} \ar[d]^{V^{\nu_p(b)}} \\
 \lim\limits_R \bigoplus\limits_{a \mid m} \TR^{\nu_p(m)-\nu_p(a)}_{q-\lambda(a,b,m)}(\bZ; p)_{(p)} \ar[r]^-{V^{\nu_p(a)}} &  \lim\limits_R \bigoplus\limits_{m} \TR^{\nu_p(m)}_{q-\lambda(a,b,m)}(\bZ; p)_{(p)}
} \]
Now, by \cite[Theorem B]{AnGeHe09} we can compute everything up to extensions of finite abelian $p$-groups. In particular, the Verschiebung maps are all injective and when the groups are free abelian the cokernel is also free abelian.

The remainder of the proof is purely combinatial, and is very similar to the proof of the main result in \cite{AnGeHe09}. We repeat \cite[Proposition 2.1 and 2.2]{AnGeHe09}, but with different representation spheres.

In degree $q=2i$, for the lower right hand corner we end up counting the number of $j$ with $\ell(a,b,j) = i$. And for $i > 0$ we know there are precisly $ab$ such $j$. (For $i=0$ there are $ab - \frac{(a-1)(b-1)}{2}$ such $j$; this gives the smaller rank of $K_0(A, \mathfrak{a})$.) So $\rk_{\bZ_{(p)}} \lim\limits_R \bigoplus_m \TR^{\nu_p(m)}_{2i-\lambda(a,b,m)}(\bZ; p)_{(p)} = ab$. Similarly, the other corners of the diagram give rank $1$, $a$ and $b$. By using \cite[Theorem B(iii)]{AnGeHe09}, which states that in this case the Verschiebung maps are injective with cokernel a free abelian group, the statement in Theorem \ref{t:mainZ} about $K_{2i}(A,\mathfrak{a})$ follows.

In degree $q=2i-1$, for the lower right hand corner we end up counting
\[
\sum_{\ell(a,b,m) < i} \nu_p(i-\ell(a,b,m)) + \nu_p(m) 
 = \sum_{\ell(a,b,m) < i} \nu_p(i-\ell(a,b,m)) + \sum_{\ell(a,b,m) < i} \nu_p(m).
\]

We note that the set $\{ m \quad | \quad \ell(a,b,m) < i\}$ has cardinality $abi - (a-1)(b-1)/2$, and is given explicitly by
\[
 \{1,2,\ldots,abi\} \setminus \cM_i,
\]
where
\[
\cM_i = \{m \in \bN \quad | \quad m \leq abi, \ell(a,b,m) = i\}.
\]

It follows that the first sum evaluates to $\nu_p \big( \frac{i!^{ab}}{i^{(a-1)(b-1)/2}} \big)$ while the second sum evaluates to $\nu_p \big( \frac{(abi)!}{M_i} \big)$.

Similarly, the other corners of the diagram give the $p$-adic valuation of $i!^2$, $i!^a (ai)!$ and $i!^b (bi)!$, so by using \cite[Theorem B]{AnGeHe09} once more it follows that the $p$-adic valuation of $\TC_{2i-1}(A,\mathfrak{a}; p)_{(p)}$ agrees with the $p$-adic valuation of
\[
 \frac{i!^2 i!^{ab} (abi)!}{i^{(a-1)(b-1)/2} M_i i!^a(ai)! i!^b(bi)!} = \frac{i!^{(a-1)(b-1)} i! (abi)!}{i^{(a-1)(b-1)/2} M_i (ai)! (bi)!}
\]
Since this holds for every prime $p$, the result follows.
\end{proof}


\end{document}